\documentclass[12pt,reqno]{article}
\usepackage{amsmath,amssymb,amsfonts,amscd,latexsym,amsthm,mathrsfs,verbatim,comment}
\usepackage[usenames]{color}
\usepackage{enumerate}
\usepackage[unicode]{hyperref}
\newtheorem{theorem}{Theorem}[section]

\theoremstyle{remark}
\newtheorem*{remark}{\bf Remark}

\DeclareMathOperator{\asin}{asin}

\newcommand{\G}{\Gamma}

\newcommand{\z}{\zeta}

\renewcommand{\th}{\theta}
\renewcommand{\d}{{\mathrm d}}
\DeclareMathOperator{\Li}{Li}

\newcommand{\<}{\langle}
\renewcommand{\>}{\rangle}

\renewcommand{\d}{{\mathrm d}}

\newcommand{\ba}{\boldsymbol a}

\newcommand{\bc}{\boldsymbol c}
\newcommand{\fa}{\mathfrak a}
\newcommand{\fb}{\mathfrak b}
\newcommand{\fh}{\mathfrak h}

\newcommand{\fG}{\mathfrak G}
\begin{document}
\pagestyle{plain}

\title{Variations on a theme of Ap\'ery}

\author{Henri Cohen (Bordeaux) and Wadim Zudilin (Nijmegen)}

\date{3 November 2025}

\maketitle

\begin{abstract}
Ap\'ery's remarkable discovery of rapidly converging continued fractions
with small coefficients for $\zeta(2)$ and $\zeta(3)$ has led to a
wealth of important activity in an incredible variety of different directions.
Our purpose is to show that modifications of Ap\'ery's continued fractions
can give interesting results including new rapidly convergent continued
fractions for certain interesting constants.
\end{abstract}

\section{Continued fractions from Ap\'ery's acceleration}
\label{sec:intro}

Ap\'ery's remarkable discovery of rapidly converging continued fractions
(abbreviated CFs) with small coefficients for $\z(2)$ and $\z(3)$ has led to a
flurry of important activity in an incredible variety of different directions.

Our purpose here is to show that modifications of Ap\'ery's CFs
can give interesting results. The reader is warned, however, that
no new irrationality results are given\,---\,the paper is only concerned with
new elegant formulas.
Nevertheless we touch arithmetic aspects of our findings whenever appropriate.

We first recall Ap\'ery's two CFs, together with the classical one for
$\log(2)$, at the same time introducing the self-explanatory notation that we
use for CFs:

\begin{theorem}
\label{th:Ap}
We have the following CFs, together with their speed of convergence:
\begin{enumerate}
\item[$\bullet$] Tiny Ap\'ery
\begin{gather*}
\log(2)
=[[0,3(2n-1)],[2,-n^2]]
=\dfrac{2}{3-\dfrac{1}{9-\dfrac{4}{15-\dfrac{9}{21-\dfrac{16}{27-{\atop\ddots}}}}}},
\\
\log(2)-\dfrac{p(n)}{q(n)}\sim\dfrac{2\pi}{(1+\sqrt{2})^{4n+2}};
\end{gather*}
\item[$\bullet$] Small Ap\'ery
\begin{gather*}
\z(2)=[[0,11n^2-11n+3],[5,n^4]]=\dfrac{5}{3+\dfrac{1}{25+\dfrac{16}{69+\dfrac{81}{135+\dfrac{256}{223+{\atop\ddots}}}}}},
\\
\z(2)-\dfrac{p(n)}{q(n)}\sim(-1)^n\dfrac{4\pi^2}{((1+\sqrt{5})/2)^{10n+5}};
\end{gather*}
\item[$\bullet$] Big Ap\'ery
\begin{gather*}
\z(3)
=[[0,(2n-1)(17n^2-17n+5)],[6,-n^6]]
=\dfrac{6}{5-\dfrac{1}{117-\dfrac{64}{535-\dfrac{729}{1463-{\atop\ddots}}}}},
\\
\z(3)-\dfrac{p(n)}{q(n)}\sim\dfrac{4\pi^3}{(1+\sqrt{2})^{8n+4}}.
\end{gather*}
\end{enumerate}
\end{theorem}

\begin{proof}
Ap\'ery's acceleration method is now well-known at least to experts,
and its friendly description in great detail can be found in \cite{Co24a}.
It is also available in the GIT branch {\tt henri-ellCF2} of {\tt Pari/GP} as
{\tt cfapery}, which automates the method and gives the Ap\'ery accelerate,
and if desired the acceleration arrays.

Thus, it suffices to say that the above three CFs are directly obtained
by using this command (together with the {\tt cfsimplify} command) from the
standard slowly convergent series
\[
\log(2)=\sum_{k=1}^\infty\frac{(-1)^{k-1}}k, \quad
\z(2)=\sum_{k=1}^\infty\frac1{k^2}, \quad\text{and}\quad
\z(3)=\sum_{k=1}^\infty\frac1{k^3},
\]
respectively.
\end{proof}

Numerous variations on these CFs have been given in the literature, and
we again refer to \cite{Co24a} for many examples.

The two variations that we will study in this paper are first, a \emph{shift}
of the variable $n$ of the CF (i.e., replacing $n$ by $n+\alpha$ for
some $\alpha$). Shifting by an integer changes the CF into an equivalent
one, so is not interesting. The next simplest shift is with $\alpha=1/2$.

The second variation is to obtain a \emph{continuous} generalisation of the
CF, in other words a CF depending on a variable $z$ which specialises to
one of Ap\'ery's for $z=0$.

Shifting is, of course, a trivial transformation, but essentially the only
difficulty is in recognising the limit of the new CF. We do not know any
systematic way of doing that. The difficulty is compounded by the fact that
modifying a finite number of coefficients of a CF gives a limit which is
a M\"obius transformation of the initial limit, so the latter is only defined up to such a rational transformation.

The results contained in this paper have been found only because of the existence
at our disposal of a huge encyclopedia of CFs that the first author has made available
on \href{arxiv.org}{\texttt{arXiv}} as~\cite{Co24b}. We have simply ``looked'' (modulo some elementary
manipulations) at this encyclopedia to see if we could find the new CF
or something close.

\section{Small Ap\'ery}
\label{sec:sA}

Using $n\mapsto n+1/2$ on small Ap\'ery and removing denominators gives the
CF \ $S=[[0,44n^2+1],[20,(2n+1)^4]]$. Using the encyclopedia \cite{Co24b}, one immediately
finds that
$S=-10+800(\G(1/4)/\G(3/4))^4$, or equivalently:

\begin{theorem}
\label{th:sA}
We have
\begin{align*}
\left(\frac{\G(1/4)}{\G(3/4)}\right)^4
&=[[80,47,44n^2+1],[-160,(2n+1)^4]]
\\
&=80-\dfrac{160}{47+\dfrac{3^4}{44\cdot2^2+1+\dfrac{5^4}{44\cdot3^2+1+\dfrac{7^4}{44\cdot4^2+1+\dfrac{9^4}{44\cdot5^2+1+{\atop\ddots}}}}}}
\end{align*}
with speed of convergence
\[
\left(\frac{\G(1/4)}{\G(3/4)}\right)^4-\dfrac{p(n)}{q(n)}
\sim(-1)^{n+1}\dfrac{8(\G(1/4)/\G(3/4))^4}{((1+\sqrt{5})/2)^{10n+10}}.
\]
\end{theorem}

\begin{proof}
Using for instance Ramanujan's CFs for quotients of products of gamma
functions, we begin by showing that
\[
\left(\dfrac{\G(1/4)}{\G(3/4)}\right)^4=[[80,17,16n],[-64,(2n+1)^4]].
\]
Note incidentally that this is equivalent via Euler's transformation to
the hypergeometric identities:
\begin{align*}\left(\dfrac{\G(1/4)}{\G(3/4)}\right)^4&=80-\dfrac{2048}{625}\sum_{n\ge0}\dfrac{(n+1)(3/4)_n^4}{(9/4)_n^4}\\
\left(\dfrac{\G(3/4)}{\G(1/4)}\right)^4&=\dfrac{1}{16}-\dfrac{4}{81}\sum_{n\ge0}\dfrac{(2n+1)(1/4)_n^4}{(7/4)_n^4}\;,\end{align*}
where as usual $(a)_n$ denotes the rising Pochhammer symbol.

We now apply Ap\'ery's acceleration \cite{Co24a} on the above CF to obtain an
apparently complicated period~2 CF, and\,---\,the miracle occurs,
this simplifies to the CF of the theorem. To our knowledge, it is only with
this specific CF (or its equivalents) for $(\G(1/4)/\G(3/4))^4$ that the
miracle occurs.
\end{proof}

The result can be alternatively stated for the numerators $p(n)$ and denominators $q(n)$ of the convergents as the (common) difference equation
\begin{equation}
(2n+1)^2p(n+1)=(44n^2+1)p(n)+(2n-1)^2p(n-1)
\quad\text{for}\; n=1,2,\dotsc.
\label{eq1}
\end{equation}
The substitution of $n+\frac12$ for $n$ in equation \eqref{eq1} (and division of the result by~4) gives
\[
(n+1)^2p(n+3/2)=(11n^2+11n+3)p(n+1/2)+n^2p(n-1/2),
\]
which can be immediately recognised as the small Ap\'ery recursion for the ``half-shift'' $a(n)=p(n+1/2)$.

Unlike their half-brothers, the common denominator of solutions $p(n)$ and $q(n)$ of \eqref{eq1} is essentially the square of the least common multiple of $1,3,5,\dots,2n-1$, the latter growing roughly like the $n$th power of $e^4$; this is much too large compared with $((\sqrt5+1)/2)^5=e^{2.4060591\dots}$ to
be used for irrationality.

One more way of establishing Theorem~\ref{th:sA}, again motivated by the connection with small Ap\'ery, is through casting rational solutions $p(n)$ and $q(n)$ of the difference equation \eqref{eq1} (actually, their slight modification) from the Euler--Beukers integral (see \cite{Be79})
\[
I_2(n)=(-1)^n\iint\limits_{[0,1]^2}\frac{x^{n-1/2}(1-x)^{n-1/2}y^{n-1/2}(1-y)^{n-1/2}}{(1-xy)^{n+1/2}}\,\d x\,\d y.
\]
Indeed, running creative telescoping,%
\footnote{One can use the {\tt Mathematica} package {\tt HolonomicFunctions} developed by Koutchan~\cite{Ko10}
directly on the multiple integrals.
Alternatively, Nesterenko's general theorem \cite[Theorem~2]{Ne03}
converts such integrals into \emph{single} hypergeometric sums or
Barnes-type hypergeometric integrals which can be treated by
{\tt SumTools[Hypergeometric][Zeilberger]} in {\tt Maple};
there are implementations of Zeilberger's creative telescoping algorithm \cite{PWZ96}
for single sums and integrals in {\tt sage} as well.}
which is a standard satellite of Ap\'ery's acceleration and whose development \cite{PWZ96} was very much influenced by Ap\'ery's discovery~\cite{Ap79,Po79},
we see that the integrals satisfy the same recursion~\eqref{eq1} for $n=1,2,\dots$\,.
The closed form evaluation
\[
I_2(0)=\pi^2\cdot{}_3F_2\bigg( \begin{matrix} \frac12, \, \frac12, \, \frac12 \\ 1, \, 1 \end{matrix}\biggm| 1\bigg)
=16L(\eta_4^6,2)=\frac\pi{2}\,\frac{\Gamma(1/4)^2}{\Gamma(3/4)^2}
\]
follows from the Chowla--Selberg formula; here $\eta_m=\eta(m\tau)$ with $\eta(\,\cdot\,)$ standing for the Dedekind eta function.
With more (standard!) work one further finds
\[
I_2(1)=\frac\pi{4}\,\frac{\Gamma(1/4)^2}{\Gamma(3/4)^2}-20\pi\,\frac{\Gamma(3/4)^2}{\Gamma(1/4)^2},
\]
so that, more generally,
\[
I_2(n)=\tilde p(n)\cdot\pi\,\frac{\Gamma(3/4)^2}{\Gamma(1/4)^2}
-\tilde q(n)\cdot\pi\,\frac{\Gamma(1/4)^2}{\Gamma(3/4)^2}
\]
with $\tilde p(n)$ and $\tilde q(n)$ solving \eqref{eq1} as well.
Now the quantity $(\G(1/4)/\G(3/4))^4$ featured in the theorem above appears as the quotient of the two periods in the latter linear combinations.

\begin{remark}
In a somewhat different context motivated by considerations in \cite{GZ21}, the shift \eqref{eq1} of small Ap\'ery's difference equation is featured in~\cite{OS19}.
\end{remark}

\section{Big Ap\'ery}
\label{sec:bA}

If we perform the similar shift for the big Ap\'ery recursion
\begin{equation}
(n+1)^3A(n+1)=(2n+1)(17n^2+17n+5)A(n)-n^3A(n-1),
\label{eq2a}
\end{equation}
the result is the difference equation
\begin{equation}
(2n+1)^3p(n+1)=4n(68n^2+3)p(n)-(2n-1)^3p(n-1)
\label{eq2}
\end{equation}
for the numerators $p(n)$ and denominators $q(n)$ of the convergents of CF which we display below in Theorem~\ref{th:bA}.

Note that such half-shifts have been already practised in \cite{GZ21} for the big Ap\'ery numbers
\begin{equation}
A_n=\sum_{k=0}^n\binom{n}{k}^2\binom{n+k}{k}^2 \quad\text{for}\; n=0,1,2,\dots,
\label{eq:bAn}
\end{equation}
which solve the difference equation \eqref{eq2a} but \emph{a priori} only give one solution.
The recursion for $A_{n-1/2}$ in \cite{GZ21} is different from \eqref{eq2}, as it is inhomogeneous;
remarkably enough (as we will see below) the two are in many ways complementary to each other.

Following the last recipe of Section~\ref{sec:sA} we first look for an integral solution of~\eqref{eq2}; this time it is given by the Beukers triple integral \cite{Be79}
\[
I_3(n)=\iiint\limits_{[0,1]^3}\frac{x^{n-1/2}(1-x)^{n-1/2}y^{n-1/2}(1-y)^{n-1/2}z^{n-1/2}(1-z)^{n-1/2}}{(1-(1-xy)z)^{n+1/2}}\,\d x\,\d y\,\d z
\]
for $n=0,1,2,\dots$,
where creative telescoping was employed again.
The integrals of this type, with sufficiently generic (complex) exponents replacing $n\pm1/2$, are a subject of \cite{Ne03}; they can be translated into Barnes-type hypergeometric integrals, which in turn happen to be particular instances of Meijer's $G$-functions.
We take a look at a generalisation of $I_3(n)$ in Section~\ref{sec:RV}.
The integrals can be written as integrals of $_3F_2$-hypergeometric functions as well.
For example,
\begin{align*}
I_3(0)
&=\iiint\limits_{[0,1]^3}\frac{x^{-1/2}(1-x)^{-1/2}y^{-1/2}(1-y)^{-1/2}z^{-1/2}(1-z)^{-1/2}}{(1-z)^{1/2}(1+xyz/(1-z))^{1/2}}\,\d x\,\d y\,\d z
\\
&=\pi^2
\int_0^1\frac{z^{-1/2}}{1-z}\,{}_3F_2\bigg(\begin{matrix} \frac12, \, \frac12, \, \frac12 \\ 1, \, 1 \end{matrix} \biggm| \frac{-z}{1-z} \bigg)\,\d z
\\ &
=\pi^2
\int_0^{\infty}Z^{-1/2}(1+Z)^{-1/2}\,{}_3F_2\bigg(\begin{matrix} \frac12, \, \frac12, \, \frac12 \\ 1, \, 1 \end{matrix} \biggm| -Z \bigg)\,\d Z.
\end{align*}
With the help of the (standard) modular parametrisation of the underlying $_3F_2$ we recover the latter quantity as the $L$-value
\[
I_3(0)=64L(\eta_2^4\eta_4^4,3)
\]
of a Hecke $\Gamma_0(8)$ eigenform of weight~4.
This makes this story intertwine with the one in \cite{GZ21}, where the initial value $A_{-1/2}$ of the solution of inhomogeneous equation is identified as
\begin{align*}
\pi^3A_{-1/2}
&=16\pi L(\eta_2^4\eta_4^4,2)
=\pi^3\,{}_4F_3\bigg(\begin{matrix} \frac12, \, \frac12, \, \frac12, \, \frac12 \\ 1, \, 1, \, 1 \end{matrix} \biggm| 1 \bigg)
\\
&=\iiint\limits_{[0,1]^3}\frac{x^{-1/2}(1-x)^{-1/2}y^{-1/2}(1-y)^{-1/2}z^{-1/2}(1-z)^{-1/2}}{(1-xyz)^{1/2}}\,\d x\,\d y\,\d z.
\end{align*}
In other words, the two integrals for $I_3(0)$ and $\pi^3A_{-1/2}$ happen to be multiple of two periods
\begin{align*}
\omega_+
=8L(\eta_2^4\eta_4^4,3)
&= 6.9975630166806323595567578268530960\dots,
\\
\omega_-
=4\pi iL(\eta_2^4\eta_4^4,2)
&= 8.6711873312659436466050308394689215\dots\,i
\end{align*}
attached to the eigenform $\eta_2^4\eta_4^4$.
With further insight from \cite{GZ21} (complemented by details in \cite{BKSZ24}) and with harder work, again using the modular parametrisation of the underlying $_3F_2$, we find that
\[
I_3(1)=-56\omega_+-\frac32\eta_+,
\] 
where $\eta_+$ is one of the two quasiperiods
\begin{align*}
\eta_+
&=-261.3739159094042031485947045700717759\dots,
\\
\eta_-
&=-359.3354423254855950047613470695853950\dots\,i
\end{align*}
for the same eigenform, whose numerical values we borrow from \cite[Table~12]{BKSZ24}.
The discussion above shows that $I_3(n)$ is a $\mathbb Q$-linear combination of $\omega_+$ and $\eta_+$ for any $n\ge0$, thus leading to the following result:

\begin{theorem}
\label{th:bA}
We have
\begin{align*}
-3\frac{\eta_+}{\omega_+}
&=[[112,4n(68n^2+3)],[16,-(2n+1)^6]]
\\
&=112+\dfrac{16}{4\cdot(68\cdot1^2+3)-\dfrac{3^6}{8\cdot(68\cdot2^2+3)-\dfrac{5^6}{12\cdot(68\cdot3^2+3)-\dfrac{7^6}{16\cdot(68\cdot4^2+3)-{\atop\ddots}}}}}
\end{align*}
with speed of convergence
\[
-3\frac{\eta_+}{\omega_+}-\frac{p(n)}{q(n)}
\sim\frac{192\pi^3/\omega_+^2}{(1+\sqrt2)^{8n+8}}.
\]
\end{theorem}

\begin{proof}
Using $n\mapsto n+1/2$ on big Ap\'ery and removing denominators gives the
new CF \ $S=[[0,4n(68n^2+3)],[48,-(2n+1)^6]]$, whose limit is equal to
\[
0.16921170657881854838709526498834093533256251638822745276659373255666458\dots
\]
and identified with $-336-9\eta_+/\omega_+$ on the basis of the difference equation \eqref{eq2} and initial values $I_3(0),I_3(1)$ of its solution.
\end{proof}

\begin{remark}
Nesterenko's \cite[Theorem 2]{Ne03} casts the integral
\[
r(n)=\iiint\limits_{[0,1]^3}\bigg(\frac{x(1-x)y(1-y)z(1-z)}{1-(1-xy)z}\bigg)^n\,\frac{\d x\,\d y\,\d z}{1-(1-xy)z}
\]
as the Barnes-type hypergeometric integral
\[
r(n)=\frac1{2\pi i}\int_{-i\infty}^{+i\infty}\frac{\Gamma(s+n+1)^4\Gamma(-s)^2}{\Gamma(s+2n+2)^2}\,\d s,
\]
where the path of integration separates the increasing sequence of poles of $\Gamma(-s)$ from the decreasing one of $\Gamma(s+n+1)$,
for any \emph{complex} $n$ with $\operatorname{Re}(n)>-1$.
The integral can be recognised as a Meijer $G$-function,
\[
r(n)
=G^{2,4}_{4,4}\bigg(1 \biggm| \begin{matrix} -n, \, -n, \, -n, \, -n \\ 0, \, 0, \, 1-2n, \, 1-2n \end{matrix}\bigg)
=G^{4,2}_{4,4}\bigg(1 \biggm| \begin{matrix} -n, \, -n, \, n+1, \, n+1 \\ 0, \, 0, \, 0, \, 0 \end{matrix}\bigg)
\]
(see also \cite{Ne96}), and this provides us with analytic continuation of $r(n)$ to an entire function of $n\in\mathbb C$.
In analogy with \cite[Theorem 1]{GZ21}, this entire function satisfies the \emph{homogeneous} difference equation~\eqref{eq2a}
for all $n\in\mathbb C$ and is symmetric under $n\mapsto-n-1$.
Unlike the situation covered in \cite{GZ21}, it interpolates Ap\'ery's approximations to $\zeta(3)$ rather than the big Ap\'ery numbers \eqref{eq:bAn} which only satisfy the recursion for $n\in\mathbb Z$.
\end{remark}

\section{Half-shifts and Bessel moments}
\label{sec:bessel}

In the spirit of Theorem~\ref{th:sA}, we can start from the known CF
\[
L(\chi_{-3},2)=[[0,10n^2-10n+3],[2,-9n^4]],
\]
where $L(\chi_{-3},2)=\sum_{n\ge1}\big(\frac{-3}{n}\big)/n^2=1-1/2^2+1/4^2-1/5^2+\cdots$,
and again use $n\mapsto n+1/2$ and remove denominators. Using \cite{BBBC07}
and \cite{BBBG08} we find the following:

\begin{theorem}
We have
\begin{align*}
2^{-1/3}\left(\frac{\G(1/3)}{\G(2/3)}\right)^6
&=[[36,33,40n^2+2],[324,-9(2n+1)^4]]
\\
&=36+\dfrac{324}{33-\dfrac{9\cdot3^4}{40\cdot2^2+2-\dfrac{9\cdot5^4}{40\cdot3^2+2-\dfrac{9\cdot7^4}{40\cdot4^2+2-\dfrac{9\cdot9^4}{40\cdot5^2+2-{\atop\ddots}}}}}}
\end{align*}
with speed of convergence
\[
2^{-1/3}\left(\frac{\G(1/3)}{\G(2/3)}\right)^6-\dfrac{p(n)}{q(n)}
\sim\dfrac{4\cdot2^{-1/3}(\G(1/3)/\G(2/3))^6}{3^{2n+2}}.
\]
\end{theorem}

\begin{proof}
  This is proved in \cite{BBBG08} and is closely connected to even Bessel moments with three
  Bessel functions. Set
  \[
  c_{n,k}=\int_0^\infty t^kK_0^n(t)\,\d t,
  \]
  where $K_0$ is the $K$-Bessel function of index $0$. The authors show that
  \begin{align*}
    c_{3,0}&=\dfrac{\sqrt{3}\pi^3}{8}{}_3F_2\bigg( \begin{matrix} \frac12, \, \frac12, \, \frac12 \\ 1, \, 1 \end{matrix}\biggm| \dfrac{1}{4}\bigg)=\dfrac{3\G(1/3)^6}{32\cdot2^{2/3}\pi} \quad\text{and}\\
    c_{3,2}&=\dfrac{\sqrt{3}\pi^3}{288}{}_3F_2\bigg( \begin{matrix} \frac12, \, \frac12, \, \frac12 \\ 2, \, 2 \end{matrix}\biggm| \dfrac{1}{4}\bigg)=\dfrac{1}{9}c_{3,0}-\dfrac{\pi^4}{24}c_{3,0}^{-1},\end{align*}
  and the CF of the theorem is obtained from a CF for $c_{3,2}/c_{3,0}$
  proved in \cite{BBBC07}.
\end{proof}

Note in passing the connection to the unshifted CF, since the authors \cite{BBBC07} also
show that for the \emph{odd} moments with three Bessel functions we have
  \[
  c_{3,1}=\dfrac{3}{4}L(\chi_{-3},2) \quad\text{and}\quad c_{3,3}=L(\chi_{-3},2)-\dfrac{2}{3}.
  \]

  If instead we look at the Bessel moments for four Bessel functions, we start from the
  known CF
  \[
  \zeta(3)=[[0,(2n-1)(5n^2-5n+2)],[12/7,-16n^6]],
  \]
and again perform the half-shift and simplification, and we obtain the following:

\begin{theorem}
With the notation of Section~\ref{sec:bA}, we have
\begin{align*}
3\dfrac{\eta_{-}}{\omega_{-}}&=[[-128,n(20n^2+3)],[64,-(2n+1)^6]]\\
&=-128+\dfrac{64}{1\cdot(20\cdot1^2+3)-\dfrac{3^6}{2\cdot(20\cdot2^2+3)-\dfrac{5^6}{3\cdot(20\cdot3^2+3)-\dfrac{7^6}{4\cdot(20\cdot4^2+3)-{\atop\ddots}}}}}
\end{align*}
with speed of convergence
\[
3\dfrac{\eta_{-}}{\omega_{-}}-\dfrac{p(n)}{q(n)}\sim\dfrac{12\pi^3/(-\omega_{-}^2)}{2^{2n}}.
\]
\end{theorem}

\begin{proof} 
In \cite{BBBG08} the authors prove that
\[
8c_{4,2}/c_{4,0}=[[0,n(20n^3+3)],[1,-(2n+1)^6]],
\]
and that
\begin{align*}
    c_{4,0}&=\frac{\pi^4}{4}\,{}_4F_3\bigg(\begin{matrix} \frac12, \, \frac12, \, \frac12, \, \frac12 \\ 1, \, 1, \, 1 \end{matrix} \biggm| 1 \bigg), \\
    c_{4,2}&=-\frac{3\pi^2}{16}+\frac{c_{4,0}}{4}-\frac{3\pi^4}{64}\,{}_4F_3\bigg(\begin{matrix} \frac12, \, \frac12, \, \frac12, \, \frac12 \\ 1, \, 1, \, 2 \end{matrix} \biggm| 1 \bigg).
\end{align*}
In the notation of Section~\ref{sec:bA} these moments are recognised as
\[
c_{4,0}=-\pi i\omega_-=4\pi^2L(\eta_2^4\eta_4^4,2)
\quad\text{and}\quad
c_{4,2}=-\frac{\pi i}{512}(128\omega_-+3\eta_-)
\]
from which the theorem follows.\end{proof}

Once again, note in passing the connection to the unshifted CF, since the
  authors in \cite{BBBG08} also show that for the odd moments with four Bessel functions
  we have
  \[
  c_{4,1}=\dfrac{7}{8}\zeta(3) \quad\text{and}\quad c_{4,3}=\dfrac{7}{32}\zeta(3)-\dfrac{3}{16}.
  \]

\section{Rhin--Viola groups for half-shifts of Ap\'ery approximations}
\label{sec:RV}

The discussion in this section is motivated by arithmetic ``irrationality'' considerations for generalised rational approximations to $\zeta(2)$ and $\zeta(3)$ investigated in \cite{RV96,RV01}.
For convenience, we present them in reverse chronological order and follow \cite{Zu04} in the description of the corresponding groups which act on the approximations.
Though the parameters-exponents in those papers are assumed to be integers, the analytic and algebraic aspects are not affected by dropping these assumptions and treating them as real numbers.
Quite surprisingly, the arithmetic aspects adapt straightforwardly to the ``(half-)shifted'' approximations, and we highlight the related changes.

Define the triple integral
\[
I_3(a_0,a_1,a_2,a_3,a_4,a_5)
=\iiint\limits_{[0,1]^3}
\frac{x^{a_1}(1-x)^{a_4}y^{a_2}(1-y)^{a_5}z^{a_3}(1-z)^{a_4+a_5-a_3}}{(1-(1-xy)z)^{a_0}}\,
\frac{\d x\,\d y\,\d z}{1-(1-xy)z},
\]
which converges whenever the eight real quantities
\begin{gather*}
a_1, \; a_2, \; a_3, \; a_4, \; a_5, \; a_4+a_5-a_3, \; a_1+a_4+a_5-a_0-a_3, \; a_2+a_4+a_5-a_0-a_3
\end{gather*}
are larger than $-1$,
as well as $H_3(\bc)=I_3(a_0,a_1,a_2,a_3,a_4,a_5)$, where
\[
\bc=\begin{pmatrix}
c_{00} & c_{01} & c_{02} & c_{03} \\
c_{10} & c_{11} & c_{12} & c_{13} \\
c_{20} & c_{21} & c_{22} & c_{23} \\
c_{30} & c_{31} & c_{32} & c_{33}
\end{pmatrix}
=\begin{pmatrix}
a_0 & a_4+a_5-a_3 & a_1+a_4-a_0 & a_2+a_5-a_0 \\
a_3 & a_4+a_5-a_0 & a_1+a_4-a_3 & a_2+a_5-a_3 \\
a_1 & a_1+a_4+a_5-a_0-a_3 & a_4 & a_2+a_5-a_1 \\
a_2 & a_2+a_4+a_5-a_0-a_3 & a_1+a_4-a_2 & a_5
\end{pmatrix}.
\]
Then, as discussed in \cite{RV01} and \cite[Sect.~4]{Zu04}, the quantity
\begin{align*}
&
\frac{I_3(\ba)}{\begin{aligned}
& \Gamma(a_1+1)\,\Gamma(a_2+1)\,\Gamma(a_3+1)\,\Gamma(a_4+1)\,\Gamma(a_5+1)\,\Gamma(a_4+a_5-a_3+1)
\\[-1.2mm] &\quad\times \Gamma(a_1+a_4+a_5-a_0-a_3+1)\,\Gamma(a_2+a_4+a_5-a_0-a_3+1) \end{aligned}}
\\ &\qquad
=\frac{H_3(\bc)}{\Gamma(c_{10}+1)\,\Gamma(c_{20}+1)\,\Gamma(c_{30}+1)\,\Gamma(c_{01}+1)\,\Gamma(c_{21}+1)\,\Gamma(c_{31}+1)\,\Gamma(c_{22}+1)\,\Gamma(c_{33}+1)}
\end{align*}
is invariant under the action on $\bc$ of the group~$\fG_3$ generated by the following permutations:
\begin{itemize}
\item[$\bullet$] the permutations $\fa_j=\fa_{j0}$, where $j=1,2,3$, of the $j$th and 0th rows of the $4\times4$-mat\-rix $\bc$;
\item[$\bullet$] the permutation $\fb=\fb_{23}$ of the last two columns of the matrix;
\item[$\bullet$] the permutation
$\fh=(c_{00} \; c_{22})(c_{02} \; c_{20})(c_{11} \; c_{33})(c_{13} \; c_{31})$.
\end{itemize}
All these generators have order~$2$, while the order of the resulting group $\fG_3=\<\fa_1,\fa_2,\fa_3,\fb,\fh\>$ is 1920.

In the case when all the parameters $a_0,\dots,a_5$ are in $\mathbb Z+\frac12$ (in general, in $\mathbb Z+\alpha$ for a fixed $0\le\alpha<1$), any three (convergent) integrals from the 6-parameter family are related by contiguous relations;
the latter are made explicit in~\cite{RV01}.
This allows one to show that, with the three successive maxima $n_1\ge n_2\ge n_3$ (all half-integers!) in the 16-element multiset $\bc$, we have the inclusions
\[
d_{2n_1}d_{2n_2}d_{2n_3}I_3(\ba)
=d_{2n_1}d_{2n_2}d_{2n_3}H_3(\bc)\in\mathbb Z\omega_++\mathbb Z\eta_+,
\]
where the period $\omega_+$ and quasiperiod $\eta_+$ are defined in Section~\ref{sec:bA}.
Here and in what follows, $d_N$ denotes the least common multiple of the numbers $1,2,\dots,N$.
In theory, one can execute the machinery from \cite{RV01} of using the invariance of these approximations under the action of group $\fG_3$; unfortunately, this does not imply any irrationality result.

Similarly, consider the double integral
\[
I_2(a_0,a_1,a_2,a_3,a_4)
=\iint\limits_{[0,1]^2}
\frac{x^{a_1}(1-x)^{a_3}y^{a_2}(1-y)^{a_4}}
{(1-xy)^{a_0}}\,\frac{\d x\,\d y}{1-xy}
\]
whose convergence requires the five real quantities
\begin{gather*}
a_1, \; a_2, \; a_3, \; a_4, \; a_3+a_4-a_0
\end{gather*}
to be larger than $-1$,
and $H_2(\bc)=I_2(a_0,a_1,a_2,a_3,a_4)$, where this time the 10-element multiset $\bc$ is inscribed in the $4\times4$-matrix
\[
\bc=\begin{pmatrix}
c_{00} & \\
& c_{11} & c_{12} & c_{13} \\
& c_{21} & c_{22} & c_{23} \\
& c_{31} & c_{32} & c_{33}
\end{pmatrix}
=\begin{pmatrix}
a_3+a_4-a_0 & \\
& a_0 & a_1+a_3-a_0 & a_2+a_4-a_0 \\
& a_1 & a_3 & a_2+a_4-a_1 \\
& a_2 & a_1+a_3-a_2 & a_4
\end{pmatrix}.
\]
Then, as discussed in \cite{RV96} and \cite[Sect.~6]{Zu04}, the quantity
\begin{align*}
&
\frac{I_2(\ba)}{\Gamma(a_1+1)\,\Gamma(a_2+1)\,\Gamma(a_3+1)\,\Gamma(a_4+1)\,\Gamma(a_3+a_4-a_0+1)}
\\ &\qquad
=\frac{H_2(\bc)}{\Gamma(c_{00}+1)\,\Gamma(c_{21}+1)\,\Gamma(c_{31}+1)\,\Gamma(c_{22}+1)\,\Gamma(c_{33}+1)}
\end{align*}
is invariant under the action on $\bc$ of the group~$\fG_2$ generated by the following permutations:
\begin{itemize}
\item[$\bullet$] the permutations $\fa_j=\fa_{j3}$, where $j=1,2$, of the $j$th and the 3rd rows of $\bc$;
\item[$\bullet$] the permutation $\fb=\fb_{23}$ of the last two columns of $\bc$;
\item[$\bullet$] the permutation
$\fh=(c_{00} \; c_{22})(c_{11} \; c_{33})(c_{13} \; c_{31})$.
\end{itemize}
All these generators have order~$2$, while the order of the resulting group $\fG_2=\<\fa_1,\fa_2,\fb,\fh\>$ is 120.

When all the parameters $a_0,\dots,a_4$ are in $\mathbb Z+\frac12$, one can demonstrate that
\[
d_{2n_1}d_{2n_2}I_2(\ba)
=d_{2n_1}d_{2n_2}H_2(\bc)\in\mathbb Z\,\pi\,\frac{\Gamma(3/4)^2}{\Gamma(1/4)^2}+\mathbb Z\,\pi\,\frac{\Gamma(1/4)^2}{\Gamma(3/4)^2},
\]
with $n_1\ge n_2$ being two successive maxima of the 10-element multiset $\bc$,
and use the action of the group $\fG_2$ to optimise the $p$-adic size of these approximations as in~\cite{RV96}.
This does not have any irrationality consequences for the quotient $(\G(1/4)/\G(3/4))^4$ of the periods; at the same time it is already known that the latter quantity is transcendental\,---\,very much in parallel with our knowledge about $\zeta(2)=\pi^2/6$.

\section{Tiny Ap\'ery}
\label{sec:tA}

Using $n\mapsto n+1/2$ on tiny Ap\'ery and removing denominators gives the
new CF \ $S=[[0,12n],[4,-(2n+1)^2]]$. Using the encyclopedia \cite{Co24a}, one immediately
finds that
\[
S=4-32\left(\dfrac{\G(3/4)}{\G(1/4)}\right)^2\,,
\]
or equivalently:

\begin{theorem}
\label{th:tA}
We have
\[
\left(\frac{\G(1/4)}{\G(3/4)}\right)^2=[[8,11,12n],[8,-(2n+1)^2]]
=8+\dfrac{8}{11-\dfrac{9}{24-\dfrac{25}{36-\dfrac{49}{48-\dfrac{81}{60-{\atop\ddots}}}}}}
\]
with speed of convergence
\[
\left(\dfrac{\G(1/4)}{\G(3/4)}\right)^2-\dfrac{p(n)}{q(n)}
\sim\dfrac{4(\G(1/4)/\G(3/4))^2}{(1+\sqrt{2})^{4n+4}}.
\]
\end{theorem}

This CF is essentially due to Ramanujan.

\begin{proof}
We start with the following classical CF:
\[
\left(\dfrac{\G(s)}{\G(s+1/2)}\right)^2=[[0,4s-1,8s-2],[4,(2n-1)^2]]
\]
and specialise it to $s=3/4$. Inverting gives
\[
\left(\dfrac{\G(1/4)}{\G(3/4)}\right)^2=[[8,4],[4,(2n+1)^2]].
\]
Applying Ap\'ery and simplifying (in other words using
{\tt cfsimplify(cfapery())}) proves the proposition.
\end{proof}

\begin{remark}
We could have applied Ap\'ery to many other CFs for
$(\G(1/4)/\G(3/4))^2$, for instance to $[[4,1,2],[8,(2n-1)(2n+1)]]$.
However, the ``miracle'' of the specific CF that we chose is that the period~2
Ap\'ery that one obtains simplifies to the desired one. As already mentioned,
this miracle occurs quite rarely.
\end{remark}

\section{Continuous generalisations}

\subsection{Tiny Ap\'ery}

In this case there is no need to look very far:

\begin{theorem}
\label{th:logA}
We have
\[
\Li_1(z)=-\log(1-z)=[[0,(2n-1)(2-z)],[2z,-n^2z^2]]
\]
with speed of convergence
\[
\Li_1(z)-\frac{p(n)}{q(n)}\sim\frac{2\pi}{((1+\sqrt{1-z})^2/z)^{2n+1}}.
\]
\end{theorem}

\begin{proof}
The Euler transformation into a CF of the Taylor expansion of
$-\log(1-z)$ is the CF $-\log(1-z)=[[0,n+(n-1)z],[z,-n^2z]]$. Ap\'ery assumes
implicitly that $|z|$ is large, so changing $z$ into $1/z$, simplifying, applying
Ap\'ery, and changing back $z$ into $1/z$ immediately gives
$-\log(1-z)=[[0,(2n-1)(2-z)],[2z,-n^2z^2]]$, which specialises to tiny Ap\'ery
for $z=-1$ (or $z=1/2$).
\end{proof}

One more way to witness the CF in Theorem~\ref{th:logA} is by using the Euler integrals
\[
r(n)=r(n;z)=\int_0^1\frac{x^{n}(1-x)^{n}}{(1-zx)^{n+1}}\,\d x
\]
which satisfy the recursion
\[
z^2(n+1)r(n+1) - (2-z)(2n+1)r(n) + nr(n-1)=0 \quad\text{for}\; n=1,2,\dots
\]
and initial conditions $r(0)=(1/z)\Li_1(z)$, $r(1)=(1/z)^2(2/z-1)\Li_1(z) - 2/z^2$.
This naturally leads to the half-shift variation of the theorem.

\begin{theorem}
\label{th:ellA}
We have
\begin{align*}
2\frac{E(z)}{K(z)}
&=[[2-z,4n(2-z)],[-(2n+1)^2z^2]]
\\
&=2-z-\dfrac{z^2}{4(2-z)-\dfrac{9z^2}{8(2-z)-\dfrac{25z^2}{12(2-z)-\dfrac{49z^2}{16(2-z)-{\atop\ddots}}}}}
\end{align*}
with speed of convergence
\[
2\frac{E(z)}{K(z)}-\frac{p(n)}{q(n)}\sim-\frac{2\pi/K(z)^2}{((1+\sqrt{1-z})^2/z)^{2n+2}}.
\]
\end{theorem}

Here the complete elliptic integrals $K(z)$ and $E(z)$ are defined by%
\footnote{Note that this differs by $z\mapsto k^2$ from the notation
$K(k)=(\pi/2)\,{}_2F_1(1/2,1/2;1;k^2)$ and $E(k)=(\pi/2)\,{}_2F_1(-1/2,1/2;1;k^2)$.}

\[
K(z)=\frac\pi2\cdot{}_2F_1\Big(\frac12,\frac12;1;z\Big)
\quad\text{and}\quad
E(z)=\frac\pi2\cdot{}_2F_1\Big({-\frac12},\frac12;1;z\Big).
\]

\begin{proof}
Consider the integrals
\[
I_1(n;z)=\int_0^1\frac{x^{n-1/2}(1-x)^{n-1/2}}{(1-zx)^{n+1/2}}\,\d x.
\]
They satisfy the three-term recursion
\[
z^2(2n+1)I_1(n+1;z) - 4n(2-z)I_1(n;z) + (2n-1)\,I_1(n-1;z)=0
\quad\text{for} \; n=1,2,\dots,
\]
with initial values
\[
I_1(0;z)=2K(z) \quad\text{and}\quad I_1(1;z)=(2/z)(2/z-1)K(z)-(2/z)^2E(z).
\]
Simple manipulations convert this into the CF in the theorem.
\end{proof}

It follows from this proof that
\[
z^nI_1(n;z)=A_n(1/z)K(z)-B_n(1/z)E(z),
\]
where the polynomials $A_n(x)$ and $B_n(x)$ have rational coefficients and degree (at most) $n$.
For suitable rational $z$ with $|z|<1$, for example, for $z=1/q$ with $q\in\mathbb Z\setminus\{-1,0,1\}$, these approximations can be used to prove the linear independence of $K(z)$ and $E(z)$ over $\mathbb Q$,
with related estimates for the irrationality measure of $K(z)/E(z)$.
The details can be found in \cite{HM01}, where all integers $q\ne-1,2$ are treated;
it is also shown there that for the polynomials in the linear forms we have $d_{2n-1}A_n(x),d_{2n-1}B_n(x)\in\mathbb Z[x]$.

The cases $z=-1$ (or $z=1/2$), corresponding to the quantity in Theorem~\ref{th:tA},
are missed from the analysis in \cite{HM01} for clear arithmetic reasons:
the denominators $d_{2n-1}$ of $p(n)$ and $q(n)$ in these cases grow like $e^{2n+o(n)}$ as $n\to\infty$,
while the approximations are asymptotically $(1/(1+\sqrt2))^{2n+o(n)}$ and $e/(1+\sqrt2)=1.1259\hdots>1$.
The irrationality criterion from \cite{Zu17} is however applicable here, since $e^3/(1+\sqrt2)^2=3.4461\hdots<4$, thus implying the irrationality of $(\G(1/4)/\G(3/4))^2$.
This implication is of course overridden by the transcendence of the number.
Though quantitative versions of the result from \cite{Zu17} do not produce any reasonable estimates for
the irrationality exponent of $(\G(1/4)/\G(3/4))^2$, further development of the methodology from \cite{CDT24}
may be successful in this situation.

\subsection{Small Ap\'ery}

Here, by sheer luck we indeed find in the encyclopedia a continuous
generalisation of small Ap\'ery as follows:

\begin{theorem}
\label{th:cosh}
We have
\[
\dfrac{\cosh(\pi z)-1}{3z^2}=[[0,3(1-z^2),11n^2-11n+3+9z^2],[5,(n^2+4z^2)(n^2+9z^2)]]
\]
with speed of convergence
\[
\dfrac{\cosh(\pi z)-1}{3z^2}-\dfrac{p(n)}{q(n)}\sim(-1)^n\dfrac{(2/(3z^2))\sinh(2\pi z)\sinh(3\pi z)}{((1+\sqrt{5})/2)^{10n+5}}
\]
\end{theorem}

\begin{proof}
Expanding into Taylor series in $x$ the function $\cos(z\asin(x))$
and specialising, it is immediate to prove that
\[
\cos(2z\th)=\sum_{n\ge0}2^{2n}(0^2-z^2)(1^2-z^2)\cdots((n-1)^2-z^2)\dfrac{\sin^{2n}(\th)}{(2n)!}.
\]
Choosing $\th=\pi/6$ and changing $z$ into $3iz$, we obtain
\[
\cosh(\pi z)=\sum_{n\ge0}\dfrac{(0^2+9z^2)(1^2+9z^2)\cdots((n-1)^2+9z^2)}{(2n)!}.
\]
Applying Euler's transformation into a CF gives
\[
\cosh(\pi z)=[[1,2,5n^2-4n+1+9z^2],[9z^2,-2n(2n-1)(n^2+9z^2)]].
\]
We can again apply Ap\'ery, but the direct application of
the {\tt GP} function {\tt cfapery} would not work because of the introduction
of denominators in the Bauer--Muir acceleration, so one must use the more
general method explained in~\cite{Co24a}. This works with no difficulty and
proves the theorem.
\end{proof}

Same remark as above concerning this: once again we have
a ``miracle'' simplification of the period 2 Ap\'ery, which does not occur
for similar series or CFs for $\cosh(\pi z)$.

\subsection{Big Ap\'ery}

For small Ap\'ery, we were lucky to find a continuous generalization in
the encyclopedia. On the other hand, there is no such generalization of
big Ap\'ery in the encyclopedia, and we have no idea how to find one.

\section{Additional comments}
\label{sec:finale}

Making the shift $n\mapsto n+1/2$ on the CF in Theorem~\ref{th:cosh}, one obtains the
CF
\[
f(z)=[[1,12(1-z^2),44n^2+1+36z^2],[60z^2,((2n+1)^2+16z^2)((2n+1)^2+36z^2)]]
\]
which again reduces to the CF for $(\G(1/4)/\G(3/4))^4$ in Theorem~\ref{th:sA} for $z=0$.
What is this function of $z$ (if it is an interesting one)?

Here we indicate some numerical observations (keeping in mind that $f$ is only defined up to a
M\"obius transformation):
\begin{enumerate}
\item $f(iz)$ is a quasiperiodic function of period 2,
  analogue to $\cos(\pi z)$; one can easily find asymptotics for large $n$
  of $f(in)$ and of $f(i(n+1/2))$.
\item $f(z)$ is an even function of $z$, and $f''(0)=(-24S+1920)/(7S-528)$ for
  $S=(\G(1/4)/\G(3/4))^4$.
\item $f(i/2)=0$ and, more generally, if $z=i(2m+1)/4$ or $z=i(2m+1)/6$
  for integer $m\ne0$ then $f(z)$ is a rational number.
\item As $z\to\infty$ we have $f(z)=-6z^2-113/12+O(1/z^2)$, and it should be
  easy to find the asymptotic expansion to any number of terms.
\end{enumerate}

\medskip
Finally, we should mention that a different take on the shift $n\mapsto n+1/2$ led us in \cite{CZ25} to construct continued fractions for so-called Chowla--Selberg gamma quotients and to give arithmetic applications, very much in Ap\'ery's spirit.

\medskip\noindent
\textbf{Acknowledgements.}
We thank Vasily Golyshev, Albrecht Klemm, and Don Zagier for related discussions on the shifts of big Ap\'ery and quasiperiods.
We further acknowledge critical comments of two anonymous referees.


\end{document}